\documentclass{amsart}
\usepackage{graphicx}
\usepackage{graphicx}
\usepackage{color}
\usepackage{times}
\usepackage{cite}
\usepackage{enumerate,latexsym}
\usepackage{latexsym}
\usepackage{amsmath,amssymb}
\usepackage{graphicx}
\usepackage{amsthm}
\usepackage{verbatim}
\newtheorem{thm}{Theorem}[section]
\newtheorem*{theorem*}{Theorem}
\newtheorem*{acknowledgement*}{Acknowledgement}
\newtheorem{cor}[thm]{Corollary}

\newtheorem{lem}[thm]{Lemma}
\newtheorem{prop}[thm]{Proposition}
\theoremstyle{definition}
\newtheorem{defn}[thm]{Definition}
\theoremstyle{remark}
\newtheorem{rem}[thm]{Remark}

\numberwithin{equation}{section}
\newcommand{\norm}[1]{\left\Vert#1\right\Vert}
\newcommand{\abs}[1]{\left\vert#1\right\vert}
\newcommand{\set}[1]{\left\{#1\right\}}
\newcommand{\Real}{\mathbb R}

\newcommand{\func}[1]{\ensuremath{\mathop{\mathrm{#1}}} }

\newcommand{\spt}[0]{\func{spt}}

\newcommand{\sing}[0]{\func{sing}}
\newcommand{\reg}[0]{\func{reg}}
\newcommand{\xX}[0]{\mathbf{x}}

\newcommand{\yY}[0]{\mathbf{y}}
\newcommand{\hH}[0]{\mathbf{H}}

\newcommand{\IV}{\mathbf{IV}}
\newcommand{\IM}{\mathcal{IM}}

\newcommand{\M}{\mathcal{M}}
\newcommand{\OO}{\mathbf{0}}

\title[A sharp lower bound for the entropy of closed hypersurfaces]{A sharp lower bound for the entropy of closed hypersurfaces up to dimension six}
\author{Jacob Bernstein}
\address{Department of Mathematics, Johns Hopkins University, 3400 N. Charles Street, Baltimore, MD 21218}
\email{bernstein@math.jhu.edu}
\author{Lu Wang}
\address{Department of Mathematics, University of Wisconsin, 480 Lincoln Drive, Madison, WI 53706}
\email{luwang@math.wisc.edu}
\thanks{The first author was partially supported by the EPSRC Programme Grant entitled ``Singularities of Geometric Partial Differential Equations'' grant number EP/K00865X/1 and by the NSF Grant DMS-1307953. The second author was partially supported by the Chapman Fellowship of Imperial College London, the AMS-Simons Travel Grant 2012-2014, and the NSF Grant DMS-1406240.}

\begin{document}
\begin{abstract}  
The entropy is a natural geometric quantity which measures the complexity of a hypersurface in $\Real^{n+1}$.  It is non-increasing along the mean curvature flow and so plays a significant role in analyzing the dynamics of this flow. In \cite{CIMW}, Colding-Ilmanen-Minicozzi-White showed that within the class of closed smooth self-shrinking solutions of the mean curvature flow in $\Real^{n+1}$, the entropy is uniquely minimized at the round sphere. They conjectured that, for $2\leq n\leq 6$, the round sphere minimizes the entropy among all closed hypersurfaces. Using an appropriate weak mean curvature flow, we prove their conjecture. For these dimensions, our approach also gives a new proof of the main result of \cite{CIMW} and extends its conclusions to compact singular self-shrinking solutions. 
\end{abstract}
\maketitle

\section{Introduction}\label{Intro}
On $\Real^{n+1}$, consider the Gaussian weight
\begin{equation}
\Phi(\xX)=(4\pi)^{-\frac{n}{2}}e^{-\frac{|\xX|^2}{4}}.
\end{equation}
If $\Sigma$ is a hypersurface, that is, a smooth properly embedded codimension-one submanifold of $\Real^{n+1}$, the 
\emph{Gaussian surface area} of $\Sigma$ is
\begin{equation}
\mathbf{F}[\Sigma]=\int_{\Sigma}\Phi\, d\mathcal{H}^{n},
\end{equation}
where $\mathcal{H}^n$ is $n$-dimensional Hausdorff measure. Define the \emph{entropy} of $\Sigma$ by
\begin{equation}
\lambda[\Sigma]=\sup_{(\yY,\rho)\in\Real^{n+1}\times\Real^+}\mathbf{F}\left[\rho\Sigma+\yY\right].
\end{equation}
That is, the entropy of $\Sigma$ is the supremum of the Gaussian surface area over all translations and scalings of $\Sigma$. Hence, the entropy is invariant under these symmetries. Observe that the entropy of a hyperplane is one.

For $2\leq n\leq 6$, we show that the entropy of closed  (compact and without boundary) 
hypersurfaces is minimized by round spheres, verifying a conjecture of Colding-Minicozzi-Ilmanen-White \cite[Conjecture 0.9]{CIMW}. 
Let $\mathbb{S}^n$ be the round sphere in $\Real^{n+1}$ centered at the origin $\OO$ with radius $\sqrt{2n}$.
\begin{thm}\label{MainThm}
For $2\leq n\leq 6$, if $\Sigma$ is a closed hypersurface, then $\lambda[\Sigma]\geq\lambda[\mathbb{S}^n]>1$ with equality if and only if 
$\Sigma=\rho \mathbb{S}^n+\yY$ for some $\rho>0$ and $\yY\in \Real^{n+1}$.
\end{thm}  
This result also holds for certain sets of low regularity; see Corollary \ref{MainCor}. When $n=1$, Theorem \ref{MainThm} is a direct 
consequence of the work of Grayson \cite{Grayson} and Gage-Hamilton \cite{GH} and the monotonicity of entropy under curve shortening flow. 

In \cite[Theorem 0.7]{CIMW}, Colding-Ilmanen-Minicozzi-White showed that if $\Sigma$ is a closed {self-shrinker} in $\Real^{n+1}$, then $\lambda[\Sigma]\geq\lambda[\mathbb{S}^n]$ with equality if and only if $\Sigma=\mathbb{S}^n$. Recall, a \emph{self-shrinker} is an $\mathbf{F}$-stationary hypersurface, i.e., a solution to the Euler-Lagrange equation
\begin{equation} \label{SelfshrinkerEqn}
\mathbf{H}+\frac{\xX^\perp}{2}=0.
\end{equation}
Here $\xX^\perp$ is the normal component of the position vector $\xX$ and $\mathbf{H}$ is the mean curvature vector defined by
\begin{equation}
\mathbf{H}=-H\mathbf{n}=-\operatorname{div}(\mathbf{n})\mathbf{n},
\end{equation}
where $\mathbf{n}$ is a choice of unit normal and $H$ is the mean curvature.  If $\Sigma$ solves \eqref{SelfshrinkerEqn}, then the self-similar family of hypersurfaces $\set{\sqrt{-t}\, \Sigma}_{t<0}$ is a classical solution to the mean curvature flow
\begin{equation}\label{MCFeqn}
 \left(\frac{d\xX}{dt}\right)^{\perp}=\hH,
\end{equation}
justifying the terminology. Using the  monotonicity formula of Huisken \cite{Huisken}, Colding-Minicozzi \cite{CM} observed that the entropy is non-increasing along 
solutions of \eqref{MCFeqn}. In fact, the entropy of any closed hypersurface strictly decreases unless the initial hypersurface is obtained by translating and dilating a self-shrinker. As such, the entropy yields important information about the dynamical properties of the flow. 

%

To prove Theorem \ref{MainThm}, we use properties of weak mean curvature flows starting from compact initial data. As such our argument is 
independent of \cite{CIMW}. Indeed, we give a new proof of \cite[Theorem 0.7]{CIMW} for $2\leq n\leq 6$ which also extends this result to 
compact singular self-shrinkers. The key idea of our proof is that any weak flow starting from a compact initial hypersurface must become 
extinct in finite time and, moreover, if one is careful with the choice of weak flow, then the flow forms a singularity of a special type as 
it becomes extinct. Indeed, we show that these singularities must be modeled on {collapsed} singular self-shrinkers; see Definition 
\ref{Nonterminaldefn} and Proposition \ref{NonAmpleProp}. We further show that the space of collapsed singular self-shrinkers with small 
entropy is compact and its elements have good regularity properties; see Propositions \ref{NonTermOpenCondProp} and \ref{SizesingularProp}. 
By a careful induction argument, we are able to show that the minimal entropy of collapsed singular self-shrinkers is achieved on a compact 
shrinker $\Sigma_0$; see Lemma \ref{CpctShrinkerInductionLem}. Hence, using our initial observation, we conclude that $\Sigma_0$ is {entropy 
stable} and so by \cite[Theorem 0.14]{CM} must be the round sphere when $2\leq n \leq 6$.  The upper bound for the dimension comes from our 
inability, at present, to rule out the existence of singular stable stationary cones with small entropy when $n\geq 7$.

\begin{rem}
Very recently Ketover-Zhou \cite{KZh} have developed a min-max theory for the Gaussian surface area and used it to give an alternative proof of \cite[Conjecture 0.9]{CIMW} for closed surfaces in $\mathbb{R}^3$ that are not tori. Their argument is modeled on that used by Marques-Neves \cite{MN} to show the Willmore conjecture, and so is completely different from ours.
\end{rem}

Finally, we note that in \cite[Conjecture 0.10]{CIMW}, Colding-Ilmanen-Minicozzi-White further conjectured that among all non-flat complete self-shrinkers, the round sphere has the lowest entropy. By taking a tangent flow at the first singular time of a classical mean curvature flow, one observes that, modulo regularity issues, Theorem \ref{MainThm} would follow from this stronger conjecture. Using the conclusions of this paper, we verify this conjecture for the class of \emph{partially collapsed} self-shrinkers. Essentially, a self-shrinker is partially collapsed if there is an asymptotic direction in which it is ``small''; see Definition \ref{partiallycollapsedDefn}. 
\begin{thm} \label{CollapsedShrinkers}
For $2\leq n\leq 7$, if $\Sigma$ is a complete non-compact self-shrinker in $\Real^{n+1}$ which is partially collapsed, then $\lambda[\Sigma]\geq\lambda[\mathbb{S}^{n-1}](>\lambda[\mathbb{S}^n])$ with equality if and only if, up to an ambient rotation of $\Real^{n+1}$, $\Sigma=\mathbb{S}^{n-1}\times\Real$.
\end{thm}

\section{Notation}
We will make heavy use of the results of \cite{Ilmanen1} on weak mean curvature flows. For this reason, we follow the notation of \cite{Ilmanen1} as closely as possible. 

Denote by
\begin{itemize}
 \item $\M(\Real^{n+1})=\set{\mu: \mu\mbox{ is a Radon measure on $\Real^{n+1}$}}$ (see \cite[Section 4]{Simon});
 \item $\IM_k(\Real^{n+1})=\set{\mu: \mu\mbox{ is an integer $k$-rectifiable Radon measure on $\Real^{n+1}$}}$ (see \cite[Section 1]{Ilmanen1});
 \item $\IV_k(\Real^{n+1})=\set{V: V\mbox{ is an integer rectifiable $k$-varifold on $\Real^{n+1}$}}$ (see \cite[Section 1]{Ilmanen1} or \cite[Chapter 8]{Simon}).
\end{itemize}
The space $\M(\Real^{n+1})$ is given the weak* topology. That is,
\begin{equation}
 \mu_i\to\mu\iff\int f\, d\mu_i\to\int f \, d\mu\mbox{  for all $f\in C^0_c(\Real^{n+1})$}.
\end{equation}
And the topology on $\IM_k(\Real^{n+1})$ is the subspace topology induced by the natural inclusion into $\M(\Real^{n+1})$. For the details of the topologies considered on $ \IV_k(\Real^{n+1})$, we refer to \cite[Section 1]{Ilmanen1} or \cite[Chapter 8]{Simon}. There are natural bijective maps
\begin{equation}
\begin{array}{ccc}
 V: \IM_k(\Real^{n+1})\to \IV_k(\Real^{n+1}) & \mbox{and} & \mu:\IV_k(\Real^{n+1})\to\IM_k(\Real^{n+1}).
\end{array} 
\end{equation}
The second map is continuous, but the first is not. Henceforth, write $V(\mu)=V_\mu$ and $\mu(V)=\mu_V$. 

If $\Sigma\subset\Real^{n+1}$ is a $k$-dimensional smooth properly embedded submanifold, we denote by $\mu_\Sigma=\mathcal{H}^k\lfloor\Sigma\in\IM_k(\Real^{n+1})$. Given $(\yY,\rho)\in\Real^{n+1}\times\Real^+$ and $\mu\in\IM_k(\Real^{n+1})$, we define the rescaled measure $\mu^{\yY,\rho}\in\IM_k(\Real^{n+1})$ by
\begin{equation}
 \mu^{\yY,\rho}(\Omega)=\rho^{k}\mu\left(\rho^{-1}\Omega+\yY\right).
\end{equation}
This is defined so that if $\Sigma$ is a $k$-dimensional smooth properly embedded submanifold, then
\begin{equation}
\mu^{\yY,\rho}_\Sigma=\mu_{\rho (\Sigma-\yY)}. 
\end{equation}
One of the defining properties of $\mu\in \IM_k(\Real^{n+1})$ is that for $\mu$-a.e. $\xX\in\Real^{n+1}$, there is an integer value 
$\theta_\mu(\xX)$ so that
\begin{equation}
 \lim_{\rho\to \infty}\mu^{\xX,\rho}=\theta_\mu(\xX)\mu_{P},
\end{equation}
where $P$ is a $k$-dimensional plane through the origin. When such $P$ exists, we denote it by $T_{\xX} \mu$ the \emph{approximate tangent plane at $\xX$}. The value $\theta_\mu(\xX)$ is the \emph{multiplicity of $\mu$ at $\xX$} and by definition, $\theta_\mu(\xX)\in\mathbb{N}$ for $\mu$-a.e. $\xX$. Notice that if $\mu=\mu_{\Sigma}$, then $T_{\xX}\mu=T_{\xX}\Sigma$ and $\theta_\mu(\xX)=1$. Given a $\mu\in \IM_n(\Real^{n+1})$, set
\begin{equation}
\reg(\spt(\mu))=\set{\xX\in\spt(\mu): \exists\rho>0 \mbox{ s.t. $B_\rho(\xX)\cap\spt(\mu)$ is a hypersurface}}
\end{equation}
and $\sing(\spt(\mu))=\spt(\mu)\setminus\reg(\spt(\mu))$. Here $B_\rho(\xX)$ is the open ball in $\Real^{n+1}$ centered at $\xX$ with radius $\rho$. Likewise,
\begin{equation}
\begin{array}{ccc} \reg(\mu)=\set{\xX\in \reg(\spt(\mu)): \theta_\mu(\xX)=1} & \mbox{and} & \sing(\mu)=\spt(\mu)\setminus\reg(\mu).
\end{array}
\end{equation}

For $\mu\in\IM_n(\Real^{n+1})$, we extend the definitions of $\mathbf{F}$ and $\lambda$ in the obvious manner, namely,
\begin{equation}
\mathbf{F}[\mu]=\mathbf{F}[V_\mu]=\int \Phi\, d\mu \quad \mbox{and} \quad \lambda[\mu]=\lambda[V_\mu]=\sup_{(\yY,\rho)\in\Real^{n+1}\times\Real^+}\mathbf{F}\left[\mu^{\yY,\rho}\right].
\end{equation}

Finally, we will need to consider certain oriented sets with possibly singular boundaries. For our purposes, sets of finite perimeter will suffice. Recall that a set $E$ is of \emph{locally finite perimeter}, if the characteristic function $\chi_E$ is in $BV_{loc}(\Real^{n+1})$ the space of functions with locally bounded variation; see \cite[Sections 6 and 14]{Simon}. By De Giorgi's Theorem (see \cite[Theorem 14.3]{Simon}), if $E$ is of locally finite perimeter, then the reduced boundary $\partial^\ast E$ as defined in \cite[Equation (14.2)]{Simon} is $n$-rectifiable and the total variation measure $\abs{D\chi_E}=\mathcal{H}^n\lfloor\partial^\ast E\in \IM_n(\Real^{n+1})$.
\begin{defn}
 A $\mu\in\IM_n(\Real^{n+1})$ is a \emph{compact boundary measure}, if there is a bounded open non-empty subset $E\subset\Real^{n+1}$ of locally finite perimeter so that $\spt (\mu)=\partial E$ and $\mu=\mathcal{H}^n\lfloor\partial^\ast E$. Such a set $E$ is called the \emph{interior} of $\mu$.
\end{defn}
The following approximation lemma, which is essentially \cite[Theorem 1.24]{Giusti}, will be used later in our proof of Theorem \ref{MainThm}. Since the proof of Lemma \ref{ApproxLem} is slightly technical, it may be found in Appendix \ref{ProofApprox}.
\begin{lem}\label{ApproxLem}
Let $E$ be a set of locally finite perimeter and suppose $B_{2r}(\xX)\subset E\subset B_{R}(\xX)$. Then there exists a sequence of open sets 
$E_j$ of finitely many components so that the $\partial E_j$ are hypersurfaces and:
\begin{enumerate}
\item $B_{r}(\xX)\subset E_j\subset B_{2R}(\xX)$; 
\item $\chi_{E_j}\to\chi_E$ in $L^1(\Real^{n+1})$; 
\item $\abs{D\chi_{E_j}}\to\abs{D\chi_E}$ in the sense of measures.
\end{enumerate}
\end{lem}

\section{Weak Mean Curvature Flows}\label{Wmcf}
In this section, we give a brief review of various notions of weak mean curvature flow in both the measure-theoretic and set-theoretic 
senses as well as fix our notation for them. As this section is rather technical, it may be skimmed on first reading and the expert should feel free to consult it only as needed.  
The key fact from this section that will be used in what follows is the existence of a suitable weak mean curvature flow, called a canonical boundary motion, starting from generic closed hypersurfaces.
Loosely speaking, the canonical boundary motions are distinguished by their regularity properties.  Specifically,   Huisken's monotonicity formula \cite{Huisken} and the regularity theory of 
Brakke \cite{B} apply to them and they cannot disappear suddenly.
Their existence is ensured by work of Ilmanen \cite{Ilmanen1}.

\subsection{Brakke flow}\label{Brakkef}
Historically, the first weak mean curvature flow was the measure-theoretic flow introduced by Brakke \cite{B}. This flow is called a 
\emph{Brakke flow}. Brakke's original definition considered the flow of varifolds. Here we use the (slightly stronger) notion introduced by 
Ilmanen \cite{Ilmanen1}. For our purposes, the Brakke flow has two important roles. The first is the fact that Huisken's monotonicity 
formula \cite{Huisken} holds also for Brakke flows; see \cite{Ilmanen1} or \cite{White}. The second is the powerful regularity theory of 
Brakke \cite{B} for such flows. A major technical difficulty inherent in using Brakke flows is that there is a great deal of non-uniqueness. 
Most problematic for our applications is that, by construction, Brakke flows are allowed to vanish suddenly and gratuitously.

Let $\mu\in\mathcal{M}(\Real^{n+1})$ and $\phi\in C^2_c(\Real^{n+1},\Real^{\geq 0})$. Following \cite[Section 6.2]{Ilmanen1}, if one of the following cases happens
\begin{enumerate}
\item $\mu\lfloor\{\phi>0\}$ is not an $n$-rectifiable Radon measure;
\item $|\delta V|\lfloor\{\phi>0\}$ is not a Radon measure on $\{\phi>0\}$, where $V=V_{\mu}\lfloor\{\phi>0\}$;
\item $|\delta V|\lfloor\{\phi>0\}$ is singular with respect to $\mu\lfloor\{\phi>0\}$;
\item $\int\phi |\hH|^2\; d\mu=\infty$, where $\mathbf{H}=d(\delta V)/d\mu\lfloor\{\phi>0\}$,
\end{enumerate}
then take $\mathcal{B}(\mu,\phi)=-\infty$. Otherwise, let
\begin{equation}
\mathcal{B}(\mu,\phi)=\int -\phi H^2+D\phi\cdot S^\perp\cdot\mathbf{H}\, d\mu,
\end{equation}
where $S=S(\xX)=T_{\xX}\mu$ for $\mathcal{H}^n$-a.e. $\xX\in\{\phi>0\}$ and $S^\perp\cdot \yY$ is understood to mean to project the vector $\yY$ onto the line $S^\perp$ perpendicular to $S$. 

Let $I$ denote an interval in $\Real$. The \emph{upper derivative} of a function $f:I\to\Real$ is
\begin{equation}
\bar{D}_{t_0}f(t)=\limsup_{t\in I, t\to t_0}\frac{f(t)-f(t_0)}{t-t_0}.
\end{equation}

A family $\mathcal{K}=\set{\mu_t}_{t\in I}$ with $\mu_t\in\mathcal{M}(\Real^{n+1})$ is a \emph{codimension-one Brakke flow in $\Real^{n+1}$}, if for all $t\in I$ and $\phi\in C^2_c(\Real^{n+1},\Real^{\geq 0})$,
\begin{equation} \label{BrakkeIneq}
\bar{D}_t\mu_t(\phi)\leq\mathcal{B}(\mu_t,\phi),
\end{equation}
or equivalently, for all $t_1,t_2\in I$ with $t_1\leq t_2$ and $\phi\in C^1_c(\Real^{n+1},\Real^{\geq 0})$,
\begin{equation}
\int\phi \, d\mu_{t_2}\leq\int\phi \, d\mu_{t_1}+\int_{t_1}^{t_2}\mathcal{B}(\mu_t,\phi)\, dt.
\end{equation}
Given $U\subset\Real^{n+1}$ a non-empty open subset and a subinterval $I^\prime\subset I$, we could restrict $\mathcal{K}$ to $U\times I^\prime$ by
\begin{equation}
\mathcal{K}\lfloor U\times I^\prime=\set{\mu_t\lfloor U}_{t\in I^\prime},
\end{equation}
which clearly satisfies the inequality \eqref{BrakkeIneq} for all $\phi\in C^2_c(U,\Real^{\geq 0})$. When the meaning is clear from context, we will suppress mention of the codimension and ambient domain and speak of a \emph{Brakke flow}. A Brakke flow, $\mathcal{K}=\set{\mu_t}_{t\in I}$ is \emph{integral} if $\mu_t\in \IM_n(\Real^{n+1})$ for a.e. $t\in I$. 

We will generally restrict our attention to Brakke flows $\mathcal{K}=\set{\mu_t}_{t\geq{t_0}}$ with bounded area ratios, i.e., for which there is a $C<\infty$ so that for all $t\geq t_0$, 
\begin{equation}\label{ArearatioEqn}
\sup_{\xX\in\Real^{n+1}}\sup_{R>0}\frac{\mu_t\left(B_R(\xX)\right)}{R^n}\leq C.
\end{equation}

Ilmanen \cite{Ilmanen2} observed that the monotonicity formula of Huisken \cite{Huisken} could be extended to hold for Brakke flows with initial data satisfying \eqref{ArearatioEqn}; see also\cite[Section 10]{White}.
\begin{prop}[{\cite[Lemma 7]{Ilmanen2}}]\label{MonotoneProp}
Given $\mathcal{K}=\set{\mu_t}_{t\geq t_0}$ a Brakke flow, suppose that $\mu_{t_0}$ satisfies \eqref{ArearatioEqn}. Then for any $(\yY,s)$ in $\Real^{n+1}\times(t_0,\infty)$ and all $t_0\leq t_1\leq t_2<s$,
\begin{equation}
\begin{split}
& \int\Phi_{(\yY,s)} (\xX,t_2)\, d\mu_{t_2}(\xX)-\int\Phi_{(\yY,s)} (\xX,t_1)\, d\mu_{t_1}(\xX)\\
\leq & -\int_{t_1}^{t_2}\int\Phi_{(\yY,s)}(\xX,t)\left\vert\mathbf{H}(\xX,t)+\frac{S^\perp (\xX,t)\cdot(\xX-\yY)}{2(s-t)}\right\vert^2\, d\mu_t(\xX)dt,
\end{split}
\end{equation}
where $S(\xX,t)=T_{\xX}\mu_t$ and 
\begin{equation}
\Phi_{(\yY,s)}(\xX,t)=(s-t)^{-\frac{n}{2}}\Phi\left(\frac{\xX-\yY}{\sqrt{s-t}}
\right).
\end{equation}
\end{prop}
An easy, but useful, consequence of Proposition \ref{MonotoneProp} is that
\begin{cor}\label{ArearatioCor}
If $\mathcal{K}=\set{\mu_t}_{t\geq t_0}$ is a Brakke flow and $\mu_{t_0}$ satisfies \eqref{ArearatioEqn}, then $\mathcal{K}$ has bounded area ratios.
\end{cor}
Another important consequence of Proposition \ref{MonotoneProp} is that if a Brakke flow $\mathcal{K}=\set{\mu_t}_{t\geq t_0}$ has bounded area ratios, then $\mathcal{K}$ has a well defined \emph{Gaussian density} at every point $(\yY,s)\in\Real^{n+1}\times (t_0,\infty)$ defined by
\begin{equation}
\Theta_{(\yY,s)}(\mathcal{K})=\lim_{t\to s^-}\int\Phi_{(\yY,s)}(\xX,t)\, d\mu_t(\xX).
\end{equation}
It is easy to see from the monotonicity formula that $\Theta_{(\yY,s)}(\mathcal{K})\geq\theta_{\mu_s}(\yY)$. Moreover, the Gaussian density is upper semi-continuous:
\begin{cor}\label{UpperctsCor}
If $\mathcal{K}=\set{\mu_t}_{t\geq t_0}$ is a Brakke flow with bounded area ratios, then the map $(\yY,s)\mapsto\Theta_{(\yY,s)}(\mathcal{K})$ is upper semi-continuous on $\Real^{n+1}\times(t_0,\infty)$.
\end{cor}

Following Ilmanen \cite[Section 7]{Ilmanen1}, we say that a sequence of Brakke flows $\mathcal{K}^i=\set{\mu_t^i}_{t\geq t_0}$ converges to a Brakke flow $\mathcal{K}=\set{\mu_t}_{t\geq t_0}$, if
\begin{enumerate}
 \item $\mu^i_t\to \mu_t$ for all $t\geq t_0$; and
 \item for a.e. $t\geq t_0$, there is a subsequence $i_k$, depending on $t$, so that $V_{\mu^{i_k}_t}\to V_{\mu_t}$.
\end{enumerate}
Convergence for flows with varying initial times is defined analogously.

Based on the idea of Brakke \cite[Chapter 4]{B}, Ilmanen proved the following compactness theorem for integral Brakke flows.
\begin{thm}[{\cite[Theorem 7.1]{Ilmanen1}}]\label{BrakkeFlowCompactThm}
Let $\mathcal{K}^i=\set{\mu_t^i}_{t\geq t_0}$ be a sequence of integral Brakke flows so that for all bounded open $U\subset\Real^{n+1}$,
\begin{equation}
 \sup_{i}\sup_{t\in [t_0,\infty)}\mu^i_t(U)\leq C(U)<\infty.
\end{equation}
There is a subsequence $i_k$ and an integral Brakke flow $\mathcal{K}=\set{\mu_t}_{t\geq t_0}$ so that $\mathcal{K}^{i_k}\to\mathcal{K}$.
\end{thm}

Combining the compactness of Brakke flows with the monotonicity formula, one establishes the existence of tangent flows. For a Brakke flow $\mathcal{K}=\set{\mu_t}_{t\geq t_0}$ and a point $(\yY,s)\in\Real^{n+1}\times(t_0,\infty)$, define a new Brakke flow 
\begin{equation}
\mathcal{K}^{(\yY,s),\rho}=\set{\mu_t^{(\yY,s),\rho}}_{t\geq\rho^2(t_0-s)}
\end{equation}
where
\begin{equation}
\mu_t^{(\yY,s),\rho}=\mu_{s+\rho^{-2}t}^{\yY,\rho}.
\end{equation}
\begin{defn}
Let $\mathcal{K}=\set{\mu_t}_{t\geq t_0}$ be an integral Brakke flow with bounded area ratios. A non-trivial Brakke flow 
$\mathcal{T}=\set{\nu_t}_{t\in\Real}$ is a \emph{tangent flow} to $\mathcal{K}$ at $(\yY,s)\in\Real^{n+1}\times(t_0,\infty)$, if there is a 
sequence $\rho_i\to\infty$ so that $\mathcal{K}^{(\yY,s),\rho_i}\to\mathcal{T}$. Denote by $\mathrm{Tan}_{(\yY,s)}\mathcal{K}$ the set of 
tangent flows to $\mathcal{K}$ at $(\yY,s)$.
\end{defn}
The monotonicity formula implies that any tangent flow is backwardly self-similar.
\begin{thm}[{\cite[Lemma 8]{Ilmanen2}}]\label{BlowupsThm}
Given an integral Brakke flow $\mathcal{K}=\set{\mu_t}_{t\geq t_0}$ with bounded area ratios, a point $(\yY,s)\in\Real^{n+1}\times({t_0},\infty)$ with $\Theta_{(\yY,s)}(\mathcal{K})\geq 1$, and a sequence $\rho_i\to\infty$, there exists a subsequence $\rho_{i_j}$ and a $\mathcal{T}\in\mathrm{Tan}_{(\yY,s)}\mathcal{K}$ so that $\mathcal{K}^{(\yY,s),\rho_{i_j}}\to\mathcal{T}$.

Furthermore, $\mathcal{T}=\set{\nu_t}_{t\in\Real}$ is backwardly self-similar with respect to parabolic rescaling about $(\OO,0)$. That is, for all $t<0$ and $\rho>0$, 
\begin{equation}
\nu_t=\nu_t^{(\OO,0),\rho}. 
\end{equation}
Moreover, $V_{\nu_{-1}}$ is a stationary point of the $\mathbf{F}$ functional and 
\begin{equation}
\Theta_{(\yY,s)}(\mathcal{K})=\mathbf{F}[\nu_{-1}].
\end{equation}
\end{thm}

In \cite[Theorem 6.12]{B}, Brakke established a partial regularity theorem for so-called unit density Brakke flows. Brakke's proof is very difficult, however, White \cite{WhiteReg} has given an elementary proof for a special, but large, class of Brakke flows. The interested reader may verify that, for the purposes of proving Theorem \ref{MainThm}, this class suffices. The reader may also consult the recent papers \cite{KT} and \cite{T} which give a proof using the monotonicity formula.  We will use only the following consequence of Brakke's local regularity theorem \cite[Lemma 6.11]{B}; see also \cite[Theorem 12.1]{Ilmanen1}. The proof is essentially the same as that of \cite[Lemma 2.1]{Schulze}; see also  \cite[Corollary 3.4 and Theorem 7.3]{WhiteReg}.
\begin{prop}\label{BrakkeRegProp}
Let $\mathcal{K}^i=\set{\mu_t^i}_{t\geq t_0}$ be a sequence of integral Brakke flows converging to an integral Brakke flow $\mathcal{K}=\set{\mu_t}_{t\geq t_0}$. If $\mathcal{K}\lfloor B_{R}(\yY)\times (t_1,t_2)$ is a regular flow, then 
\begin{enumerate}
\item for each $t_1<t<t_2$, $\spt(\mu_t^i)\to\spt(\mu_t)$ in $C^\infty_{loc}(B_{R}(\yY))$;
\item given $\epsilon>0$, there is an $i_0=i_0(\epsilon,\mathcal{K})$ so that if $i>i_0$, $\mathcal{K}^i\lfloor B_{R-\epsilon}(\yY)\times (t_1+\epsilon,t_2)$ is a regular flow. 
\end{enumerate}
\end{prop}
Here we say an integral Brakke flow $\mathcal{K}=\set{\mu_t}_{t\in I}$ is regular, if $\mathcal{K}=\set{\mu_{\Sigma_t}}_{t\in I}$ where $\set{\Sigma_t}_{t\in I}$ is a proper smooth embedded mean curvature flow.

\subsection{Level-set flow}\label{Lsf}
We will also need a set-theoretic weak mean curvature flow called the level-set flow. This flow was first studied in the context of numerical analysis by Osher-Sethian \cite{OS}. The mathematical theory was developed by Evans-Spruck \cite{ES1,ES2,ES3,ES4} and Chen-Giga-Goto \cite{CGG}. For our purposes, it has the important advantages of being uniquely defined and satisfying a nice maximum principle. 

We will follow the formulation of the level-set flow of Evans-Spruck \cite{ES1}. Let $\Gamma$ be a compact non-empty subset of $\Real^{n+1}$. Select a continuous function $u_0$ so that $\Gamma=\set{\xX: u_0(\xX)=0}$ and there are constants $C, R>0$ so that 
\begin{equation}
u_0= -C \quad\mbox{on $\set{\xX\in\Real^{n+1}: \abs{\xX}\geq R}$}
\end{equation}
for some sufficiently large $R$. In particular, $\set{u_0\geq a>-C}$ is compact. In \cite{ES1}, Evans-Spruck established the existence and uniqueness of viscosity weak solutions to the initial value problem: 
\begin{equation}\label{LevelsetEqn}
\left\{ \begin{array}{cl} 
u_t=\sum_{i,j=1}^{n+1}\left(\delta_{ij}-u_{x_i}u_{x_j}|Du|^{-2}\right)u_{x_ix_j} & \mbox{on $\Real^{n+1}\times (0,\infty)$}\\
u=u_0  & \mbox{on $\Real^{n+1}\times\{0\}$}. 
\end{array} \right.
\end{equation}
Setting $\Gamma_t=\set{\xX: u(\xX,t)=0}$, define $\mathcal{L}(\Gamma)=\set{\Gamma_t}_{t\geq 0}$ to be the \emph{level-set} flow of $\Gamma$. This is justified by \cite[Theorem 5.3]{ES1}, which shows that  $\mathcal{L}(\Gamma)$  is independent of the choice of $u_0$.

Level-set flows satisfy an avoidance principle, namely,
\begin{prop}[{\cite[Theorem 7.3]{ES1}}]\label{AvoidProp}
Let $\mathcal{L}(\Gamma)=\set{\Gamma_t}_{t\geq 0}$ and $\mathcal{L}(\Gamma^\prime)=\set{\Gamma^\prime_t}_{t\geq 0}$ be level-set flows. Assume that $\Gamma$ and $\Gamma^\prime$ are disjoint compact non-empty subsets. Then the distance between $\Gamma_t$ and $\Gamma_t^\prime$ is non-decreasing in $t$.
\end{prop}

A technical feature of the level-set flow is that the $\Gamma_t$ of $\mathcal{L}(\Gamma)$ may develop non-empty interiors for positive times. This phenomena is called fattening and is unavoidable for certain initial sets $\Gamma$. It is closely related to non-uniqueness phenomena of weak solutions of the flow.  A level-set flow $\mathcal{L}(\Gamma)=\set{\Gamma_t}_{t\geq 0}$ is \emph{non-fattening}, if each $\Gamma_t$ has no interior.  

\subsection{Boundary motion}
In \cite{Ilmanen1}, Ilmanen synthesized both notions of weak flow. In particular, he showed that for a large class of initial sets, there is a canonical way to associate a Brakke flow to the level-set flow, and observed that this allows, among other things, for the application of Brakke's partial regularity theorem. For our purposes, it is important that the Brakke flow constructed does not vanish gratuitously. A similar synthesis may be found in \cite{ES4}.

Following \cite[Section 11]{Ilmanen1}, we introduce the following definition:
\begin{defn}
Given a compact boundary measure $\mu_0$ with interior $E_0$, a \emph{canonical boundary motion of $\mu_0$} is a pair $(E, \mathcal{K})$ consisting of an open bounded subset $E$ of $\Real^{n+1}\times\Real^{\geq 0}$ of finite perimeter and a Brakke flow $\mathcal{K}=\set{\mu_t}_{t\geq 0}$ so that: 
\begin{enumerate}
 \item $E=\set{(\xX,t): u(\xX,t)>0}$, where $u$ solves \eqref{LevelsetEqn} with $E_0=\set{\xX: u_0(\xX)>0}$ and $\partial E_0=\set{\xX: u_0(\xX)=0}$;
 \item each $E_t=\set{\xX: (\xX,t)\in E}$ is of finite perimeter and $\mu_t=\mathcal{H}^n\lfloor\partial^\ast E_t$.
\end{enumerate}
\end{defn}
A canonical boundary motion in our sense is automatically a boundary motion as considered in \cite{Ilmanen1}, however, the converse need not be true. Ilmanen showed that under a non-fattening condition on $\mathcal{L}(\spt(\mu_0))$, there always exists a corresponding boundary motion. In fact, his proof gives the existence of a canonical boundary motion of $\mu_0$. 
\begin{thm}[{\cite[Theorem 11.4]{Ilmanen1}}] \label{UnitdensityThm}
If $\mu_0$ is a compact boundary measure such that the level-set flow $\mathcal{L}(\spt(\mu_0))$ is non-fattening, then there is a canonical boundary motion $(E,\mathcal{K})$ of $\mu_0$.  
\end{thm}
It is relatively straightforward to see that the non-fattening condition is generic; see for instance \cite[Theorem 11.3]{Ilmanen1}.  

\section{Regularity and Asymptotic Structure for Self-shrinking Measures of Low Entropy}\label{Selfshrinker}
Let us define the set of self-shrinking measures on $\Real^{n+1}$ by
\begin{equation}
\mathcal{SM}_n=\set{\mu\in\IM_n(\Real^{n+1}): V_\mu\mbox{ is stationary for the $\mathbf{F}$ functional}, \spt(\mu)\neq\emptyset}.
\end{equation}
Denote by $\mathcal{CSM}_n$ the set of self-shrinking measures on $\Real^{n+1}$ with compact support. Further, given $\Lambda>0$, set
\begin{equation}
\mathcal{SM}_n(\Lambda)=\set{\mu\in\mathcal{SM}_n: \lambda[\mu]<\Lambda}
\end{equation}
and $\mathcal{CSM}_n(\Lambda)=\mathcal{CSM}_n\cap\mathcal{SM}_n(\Lambda)$.

Recall that an important class of self-shrinkers are the generalized cylinders
\begin{equation}
 \mathbb{S}^{n-k}\times\Real^{k}=\set{\sum_{i=k}^n x_{i+1}^2 = 2(n-k)}\subset \Real^{n+1},
\end{equation}
where $0\leq k\leq n$. As computed by Stone \cite{Stone},
\begin{equation}\label{StoneCompEqn}
 2>\mathbf{F}[\mathbb{S}^1]>\frac{3}{2}>\mathbf{F}[\mathbb{S}^2]>\cdots >\mathbf{F}[\mathbb{S}^n]>\cdots >1.
\end{equation}
By \cite[Lemma 7.10]{CM}, $\lambda_n=\lambda[\mathbb{S}^n]=\mathbf{F}[\mathbb{S}^n]$ and thus the same inequalities hold for $\lambda_n$.

\subsection{Regularity of self-shrinking measures of small entropy}
We begin by estimating the size of the singular set of self-shrinking measures with low entropy. In order to accomplish this, we will need a stratification result from  \cite{White}. 

A $\mu\in\IM_n(\Real^{n+1})$ is a \emph{cone}, if $\mu^{\OO,\rho}=\mu$. Likewise, $\mu\in\IM_n(\Real^{n+1})$ \emph{splits off a line}, if, up to an ambient rotation of $\Real^{n+1}$, $\mu=\hat{\mu}\times\mu_\Real$ for $\hat{\mu}\in\IM_{n-1}(\Real^{n})$. Observe that if $\mu\in\mathcal{SM}_n$ is a cone, then $V_\mu$ is stationary (for area). Similarly, if $\mu\in\mathcal{SM}_n$ splits off a line, then $\hat{\mu}\in\mathcal{SM}_{n-1}$ and $\lambda[\mu]=\lambda[\hat{\mu}]$.


We make the following observation about two-dimensional shrinking measures.
\begin{lem}\label{2dimconeLem}
If $\mu\in\mathcal{SM}_2(3/2)$ is a cone, then, up to an ambient rotation, $\mu=\mu_{\Real^2}$.
\end{lem}
\begin{proof}
Since $V_\mu$ is stationary and $\mu\in\IM_2(\Real^3)$ with $\lambda[\mu]<3/2$, we may apply Allard's integral compactness theorem (see \cite[Theorem 42.7 and Remark 42.8]{Simon}) to conclude that given $\yY\in\sing(\mu)\setminus\set{\OO}$, there exists a sequence $\rho_i\to\infty$ so that $\mu^{\yY,\rho_i}\to\nu$ and $V_\nu$ is a stationary integral varifold. Moreover, it follows from the monotonicity formula \cite[Theorem 17.6]{Simon} that $\nu$ is a cone; see also \cite[Theorem 19.3]{Simon}.

In addition, $\mu$ is a cone and so $\nu$ splits off a line. That is, up to an ambient rotation, $\nu=\hat{\nu}\times\mu_\Real$ and $V_{\hat{\nu}}$ is a one-dimensional integral stationary cone. Thus, $\spt(\hat{\nu})$ is the union of rays starting from the origin. Moreover, by the lower semi-continuity of entropy, $\lambda[\hat{\nu}]=\lambda[\hat{\nu}\times\mu_\Real]\leq\lambda[\mu]<3/2$. This implies that there are at most two rays and the stationarity of $V_{\hat{\nu}}$ gives that the rays together form a multiplicity-one line. Hence, $\nu$ is a multiplicity-one plane. Therefore, by Allard's regularity theorem (see \cite[Theorem 24.2]{Simon}), $\sing(\mu)\subset\set{\OO}$.  Hence, as $V_\mu$ is a stationary cone, the link of $\spt(\mu)$ is smooth closed geodesic in $\mathbb{S}^2$, i.e., a great circle. Therefore, $\mu$ must be a multiplicity-one plane.
\end{proof}

We may now use a dimension reduction argument to bound the size of the singular set of a self-shrinking measure with entropy less than $3/2$.
\begin{prop}\label{SizesingularProp}
The singular set of any self-shrinking measure in $\mathcal{SM}_n(3/2)$ has Hausdorff dimension at most $n-3$.\footnote{Though it is not needed in this paper, we observe that by the recent resolution of the Willmore conjecture by Marques-Neves \cite{MN}, the Hausdorff dimension estimate in this proposition can be improved to $n-4$.}
\end{prop}
\begin{proof}
Given $\mu\in\mathcal{SM}_n(3/2)$, the mean curvature of $V_\mu$ is locally bounded by \eqref{SelfshrinkerEqn}. Following the same reasoning in the proof of Lemma \ref{2dimconeLem}, given $\yY\in\sing (\mu)$, there exists a sequence $\rho_i\to\infty$ so that $\mu^{\yY,\rho_i}\to\nu$ and $V_\nu$ is an integral stationary cone. By the lower semi-continuity of entropy, $\lambda[\nu]\leq\lambda[\mu]<3/2$. Hence, together with Lemma \ref{2dimconeLem}, it follows from general dimension reduction arguments (see for instance \cite[Theorem 4]{White}) that the Hausdorff dimension of $\sing(\mu)$ is at most $n-3$.
\end{proof}


An important consequence of Proposition \ref{SizesingularProp} is that self-shrinking measures with compact support and of small entropy are compact boundary measures.    
\begin{prop}\label{OrientabilityLowEnt} 
If $\mu\in\mathcal{CSM}_n(3/2)$, then $\mu$ is a compact boundary measure.  
\end{prop}
\begin{proof}
It suffices to show that the regular part $\reg(\mu)$ is orientable. Let $\gamma$ be any closed simple curve in $\Real^{n+1}$. Then $\gamma$ bounds a topological disk $D$ with $\partial D=\gamma$. Since $\mathcal{H}^{n-1}(\sing(\mu))=0$ by Proposition \ref{SizesingularProp}, one can arrange $\gamma$ and $D$ so that the closure of $D$ does not intersect $\sing(\mu)$. Thus, the orientability of $\reg(\mu)$ follows from the same arguments as in \cite{Samelson}.
\end{proof}

\subsection{Non-collapsed self-shrinking measures and flows}
We now describe the asymptotic structure of self-shrinking measures in $\mathcal{SM}_n(3/2)$. We first note for $n=2$, stratification alone gives strong control.

For $\mu\in \mathcal{SM}_n$, we define the \emph{associated Brakke flow} $\mathcal{K}=\set{\mu_t}_{t\in\Real}$ by
\begin{equation}
\mu_t= \left \{ \begin{array}{cc} 
                               0 & t\geq 0 \\
                               \mu^{\OO, \sqrt{-t}} & t<0.
                       \end{array} \right. 
\end{equation}
We prove a splitting lemma for tangent flows to self-shrinking measures at time $0$.
\begin{lem}\label{stratificationLem}
For $\mu\in\mathcal{SM}_n$ with $\lambda[\mu]<\infty$, let $\mathcal{K}$ be the associated Brakke flow to $\mu$. If $\yY\neq\OO$ with $\Theta_{(\yY,0)} (\mathcal{K})\geq 1$, and $\mathcal{T}\in\mathrm{Tan}_{(\yY, 0)} \mathcal{K}$, then $\mathcal{T}$ splits off a line, that is $\mathcal{T}=\set{\nu_t}_{t\in\Real}$, where, up to an ambient rotation, $\nu_t=\hat{\nu}_t\times\mu_\Real$ and $\set{\hat{\nu}_t}_{t\in\Real}$ is the Brakke flow associated to $\hat{\nu}_{-1}\in\mathcal{SM}_{n-1}$.
\end{lem}
\begin{proof}
Given $\yY\neq\OO$ with $\Theta_{(\yY,0)}(\mathcal{K})\geq 1$, if $\mathcal{T}\in\mathrm{Tan}_{(\yY,0)}\mathcal{K}$, there exists a sequence $\rho_i\to\infty$ such that $\mathcal{K}^{(\yY,0),\rho_i}\to\mathcal{T}$. Thus, it follows from the self-similarity of $\mathcal{K}$ that $\mathcal{T}$ is translation invariant along the direction of $\yY$.  Indeed, for any $\tau\in \Real$, 
\begin{align*}
\mathcal{T}^{(\tau \yY,0),1} &= \lim_{i\to \infty} \mathcal{K}^{((1+\rho^{-1}_i\tau)\yY, 0), \rho_i}=\lim_{i\to \infty}  \left(\mathcal{K}^{(\OO, 0), (1+\rho^{-1}_i\tau)^{-1}}\right)^{(\yY, 0), \rho_i(1+\rho^{-1}_i\tau)}\\
&= \lim_{\rho_i\to \infty} \mathcal{K}^{(\yY, 0), \rho_i(1+\rho_i^{-1}\tau)}
\end{align*}
where we used that $\mathcal{K}$ was self-similar about $(\OO, 0)$ for the last equality. As 
$$
\lim_{i \to \infty} \frac{\rho_i(1+\rho^{-1}_i\tau)}{\rho_i}=1,
$$
we conclude that $\mathcal{T}^{(\tau \yY,0),1}=\mathcal{T}$, that is, the flow splits off a line in the direction of $\yY$.
\end{proof}

We now use this lemma together with the fact that the only one-dimensional self-shrinkers are the circle and straight lines, to conclude that any self-shrinking measure on $\Real^3$ with small entropy and non-compact support must be asymptotic to a regular cone.  The main idea is that the associated Brakke flow to the self-shrinking measure encodes at time $0$ its asympotic behavior. The self-similarity of the flow and the non-compact support of the measure imply this is a non-empty cone. Lemma \ref{stratificationLem}, the classification of the one-dimensional self-shrinkers and Brakke's regularity theorem ensure it is a regular cone.
\begin{prop}\label{AsymptoticsProp}
If  $\mu\in\mathcal{SM}_2(3/2)$ has non-compact support, then $\mu=\mu_\Sigma$, where $\Sigma$ is a smooth self-shrinking surface asymptotic at infinity to a regular cone in the strong blow-down sense. In particular, $\Sigma$ has quadratic curvature decay, i.e., for $\xX\in\Sigma$ outside some compact set,
\begin{equation}
\abs{A_\Sigma}(\xX)\leq\frac{C_0}{\abs{\xX}}
\end{equation}
for some positive constant $C_0$.  
\end{prop}
\begin{proof}
First it follows from Proposition \ref{SizesingularProp} that $\mu=\mu_\Sigma$ for a non-compact self-shrinker $\Sigma$. Moreover, the entropy bound gives that $\Sigma$ has at most quadratic area growth and thus, by \cite[Theorem 1.3]{CZProper}, $\Sigma$ is proper.

Let $\mathcal{K}=\set{\mu_t}_{t\in\Real}$ be the Brakke flow associated to $\mu$. Note that $\mu_t=\mu_{\sqrt{-t}\, \Sigma}$ for $t<0$. Let $X=\set{\yY: \yY\neq\OO, \Theta_{(\yY,0)}(\mathcal{K})\geq 1}\subset \Real^{3}\setminus\set{\OO}$. As $\Sigma$ is non-compact, $X$ is non-empty. Indeed, pick any sequence of points $\yY_i\in\Sigma$ with $|\yY_i|\to\infty$. The points $\hat{\yY}_i=|\yY_i|^{-1}\yY_i\in |\yY_i|^{-1}\Sigma$. Hence, $\Theta_{(\hat{\yY}_i, -|\yY_i|^{-2})}(\mathcal{K})\geq 1$. As the $\hat{\yY}_i$ are in a compact subset, up to passing to a subsequence and relabeling, $\hat{\yY}_i\to\hat{\yY}$, and so the upper semi-continuity of Gaussian density, Corollary \ref{UpperctsCor}, implies that $\Theta_{(\hat{\yY},0)}(\mathcal{K})\geq 1$.  

We next show that $X$ is a smooth properly embedded cone in $\Real^{3}\backslash \set{\OO}$. Given $\yY\in X$ and $\rho>0$, invoking the self-similarity of $\mathcal{K}$ about $(\OO,0)$, 
\begin{align*}
\Theta_{(\rho\yY,0)}(\mathcal{K}) &=\lim_{s\to 0-} \frac{-1}{4\pi s}\int e^{\frac{|\xX-\rho\yY|^2}{4s}}\, d\mu_s(\xX)=\lim_{s\to 0-} \frac{-\rho^2}{4\pi s} \int e^{\frac{\rho^2|\xX-\yY|^2}{4s}}\, d\mu_s^{\OO,\rho^{-1}}(\xX) \\
&=\lim_{s\to 0-} \frac{-1}{4\pi s\rho^{-2}} \int e^{\frac{|\xX-\yY|^2}{4s\rho^{-2}}} \, d\mu_{s\rho^{-2}}(\xX)=\Theta_{(\yY,0)}(\mathcal{K})\geq 1.
\end{align*}
In particular, $\rho\yY\in X$. Thus $X$ is invariant under dilation and so is a cone. To see that $X$ is regular, we note that by Lemma \ref{stratificationLem}, for any $\yY\in X$ and $\mathcal{T}\in\mathrm{Tan}_{(\yY,0)}\mathcal{K}$, $\mathcal{T}=\set{\nu_t}_{t\in\Real}$ splits off a line. That is, up to an ambient rotation, $\nu_{t}=\hat{\nu}_t\times\mu_\Real$ with $\set{\hat{\nu}_t}_{t\in\Real}$ the Brakke flow associated to $\hat{\nu}_{-1}\in\mathcal{SM}_1(3/2)$.  Here we use the lower semi-continuity of entropy. By Proposition \ref{SizesingularProp}, $\hat{\nu}_{-1}=\mu_\gamma$ for $\gamma$ a one-dimensional complete self-shrinker. Thus, by the classification theorem in \cite{AL} and inequality \eqref{StoneCompEqn}, $\hat{\nu}_{-1}$ is a multiplicity-one line and so $\nu_{-1}$ is a multiplicity-one plane and $\mathcal{T}$ is a static multiplicity-one plane. Hence, it follows from Proposition \ref{BrakkeRegProp} that for all $t<0$ close to $0$, $\spt(\mu_t)$ near $\yY$ is given by the connected graph of a smooth function over the same plane with uniformly bounded derivatives. Therefore, combining this with the upper semi-continuity of Gaussian density, we conclude that $\sqrt{-t}\, \Sigma\to X$ in $C^\infty_{loc}\left(\Real^3\backslash \set{\OO}\right)$, as $t\to 0-$.  
\end{proof}

To understand the situation for $n>2$, we will need to introduce a much weaker, but still useful, notion. 
\begin{defn}\label{Nonterminaldefn}
A $\mu\in\mathcal{SM}_n$ is \emph{non-collapsed} if there is a $\yY\in\Real^{n+1}$ and an $R>4\sqrt{n}$ so that:
\begin{enumerate}
 \item \label{NonTermTechDefn} $\sing (\mu)\cap B_R(\yY)=\emptyset$;
 \item $\spt (\mu)$ separates $B_R(\yY)\subset\Real^{n+1}$ into two components $\Omega_{+}$, $\Omega_{-}$ containing, respectively, closed balls $\bar{B}_{2\sqrt{n}}(\xX_{+})$, $\bar{B}_{2\sqrt{n}}(\xX_{-})$.
\end{enumerate}
The measure $\mu$ is \emph{strongly non-collapsed} if $\mu\times \mu_{\Real^k}$ is non-collapsed for all $k\geq 0$. 
\end{defn}
Note that Condition \eqref{NonTermTechDefn} is a technical condition that is included to simplify some proofs. Observe that the definition of non-collapsed depends on the dimension $n$ in a way that ensures that if $\mu\times \mu_\Real$ is non-collapsed,  then so is $\mu$, but the converse need not hold. Thus, $\mu$ is strongly non-collapsed if and only if $\mu\times\mu_\Real$ is strongly non-collapsed. Clearly, being non-collapsed is weaker than being strongly non-collapsed which in turn is weaker than being smoothly asymptotic to a cone. Hence, Proposition \ref{AsymptoticsProp} gives that
\begin{cor}\label{NonampleCor}
If $\mu\in\mathcal{SM}_2(3/2)$ and $\spt(\mu)$ is non-compact, then $\mu$ is strongly non-collapsed.
\end{cor}
Heuristically, a comparison with shrinking spheres implies that if $\mu\in\mathcal{SM}_n$ is non-collapsed, then the Brakke flow associated 
to $\mu$ becomes extinct at time $0$ only due to sudden vanishing. However, some care is needed. For instance,  while the shrinking spoon 
(see \cite[Figure 9]{B} or \cite[Figure 6]{Ilmanen1}) is non-collapsed, any Brakke flow starting from it must become extinct by time $0$. 
Conversely, a multiplicity-two plane is collapsed, but the associated static flow never vanishes.  Nevertheless,  all self-shrinking 
measures which are compact boundary measures are collapsed. 
\begin{lem}\label{NoCpctAmpleShrinkerLem} 
If $\mu\in\mathcal{SM}_n$ is a compact boundary measure, then $\mu$ is collapsed.
\end{lem}
\begin{proof}
Let $E$ be the interior of $\mu$. If $\mu$ is non-collapsed, there is a point $\yY\in \Real^{n+1}$ and a radius $R>0$ so that $B_R(\yY)\cap \spt(\mu)$ separates $B_R(\yY)$ into two components $\Omega_-$ and $\Omega_+$ containing closed balls $\bar{B}_{2\sqrt{n}}(\mathbf{x}_{-})$ and $\bar{B}_{2\sqrt{n}}(\mathbf{x}_{+})$ respectively. Clearly, we may assume that neither $\xX_+$ nor $\xX_-$ are the origin and, up to relabeling, that $\Omega_+\subset E$. Consider $\mathcal{K}=\set{\mu_t}_{t\in\Real}$ the Brakke flow associated to $\mu$ and
\begin{equation}
\mathcal{S}=\set{\partial B_{\sqrt{2n(1-t)}}(\mathbf{x}_{+})}_{t\geq -1}
 \end{equation}
the self-shrinking spheres starting from $\partial B_{2\sqrt{n}}(\mathbf{x}_{+})$. Since $\spt(\mu_{-1})\cap B_{2\sqrt{n}}(\mathbf{x}_{+})=\emptyset$, then by the fact that the support of a Brakke flow satisfies an avoidance principle \cite[Theorem 10.6]{Ilmanen1}, $\spt(\mu_{t})\cap B_{\sqrt{2n(1-t)}}(\mathbf{x}_{+})=\emptyset$ for all $t>-1$. This leads to a contradiction.
 
Indeed, consider the ray connecting $\OO$ to $\mathbf{x}_+$. Since $\partial E=\spt (\mu)$ is bounded, there must be a point $\mathbf{x}\in\spt(\mu)$ on this ray that is further away from $\OO$ than $\mathbf{x}_+$. Hence, there is a value $\tau\in (-1,0)$ so that $\sqrt{-\tau}\, \mathbf{x}=\mathbf{x}_+$ and so $\mathbf{x}_+\in\spt(\mu_{\tau})\cap B_{\sqrt{2n(1-\tau)}}(\mathbf{x}_{+})$, yielding the claimed contradiction.
\end{proof}

Motivated by the above observation, we make the following more general definition.
\begin{defn}\label{NoncollapseDefn}
An integral Brakke flow $\mathcal{K}=\set{\mu_t}_{t\geq t_0}$ is \emph{non-collapsed at time $\tau$}, if there is a $(\yY,s)\in\Real^{n+1}\times (t_0,\tau)$, an $R>4\sqrt{n(\tau-t)}$ and an $0<\epsilon<\min\set{\tau-s, s-t_0}$ so that:
\begin{enumerate}
 \item $\mathcal{K}\lfloor B_R(\yY)\times (s-\epsilon,s+\epsilon)$ is regular;
 \item $\spt (\mu_s)$ separates $B_R(\yY)\subset\Real^{n+1}$ into two components $\Omega_{+}$, $\Omega_{-}$ containing, respectively, closed balls $\bar{B}_{2\sqrt{n(\tau-s)}}(\xX_{+})$, $\bar{B}_{2\sqrt{n(\tau-s)}}(\xX_{-})$.
\end{enumerate}
The Brakke flow $\mathcal{K}$ is \emph{strongly non-collapsed at time $\tau$}, if $\set{\mu_t\times \mu_{\Real^k}}_{t\geq t_0}$ is non-collapsed at time $\tau$ for all $k\geq 0$.
\end{defn}
Clearly, $\mu\in\mathcal{SM}_n$ is (strongly) non-collapsed if and only if the associated Brakke flow is (strongly) non-collapsed at time $0$.  Crucially, being non-collapsed at a given time is an open condition for Brakke flows.
\begin{prop}\label{NonTermOpenCondProp}
Let $\mathcal{K}=\set{\mu_t}_{t\geq t_0}$ be an integral Brakke flow with bounded area ratios. If $\mathcal{K}^i=\set{\mu_t}_{t\geq t_0}$ are integral Brakke flows converging to $\mathcal{K}$ which are collapsed at time $\tau_i>t_0$ and $\tau_i\to\tau>t_0$, then $\mathcal{K}$ is collapsed at time $\tau$.
\end{prop}
\begin{proof}
We argue by contradiction. Suppose that $\mathcal{K}$ is non-collapsed at time $\tau$. Let $(\yY,s)\in\Real^{n+1}\times (t_0,\tau)$, $R>0$ and $\epsilon>0$ be the relevant quantities from Definition \ref{Nonterminaldefn}. It's clear from the definition that we can slightly shrink $R$ and $\epsilon$ without affecting anything. Let us denote the shrunk constants by $R'$ and $\epsilon'$. By (1) in Definition \ref{NoncollapseDefn}, $\mathcal{K}\lfloor B_{R}(\yY)\times (s-\epsilon,s+\epsilon)$ is regular. Thus, by Proposition \ref{BrakkeRegProp}, for $i$ sufficiently large, $\mathcal{K}^i\lfloor B_{R'}(\yY)\times [s-\epsilon',s+\epsilon']$ is regular. Moreover, for $s\in [s-\epsilon',s+\epsilon']$, $\spt(\mu_s^i)\to\spt(\mu_s)$ in $C^\infty_{loc}(B_{R'}(\yY))$. Therefore, for $i$ sufficiently large, $\spt(\mu_s^i)$ separates $B_{R'}(\yY)$ into two components $\Omega_+^i$, $\Omega_-^i$ containing, respectively, closed balls $\bar{B}_{2\sqrt{n(\tau-s)}}(\xX_{+})$, $\bar{B}_{2\sqrt{n(\tau-s)}}(\xX_{-})$. That is, for large $i$, $\mathcal{K}^i$ are non-collapsed at time $\tau$. The size of the closed balls can always be slightly increased or decreased, and so for large $i$, the $\mathcal{K}^i$ are also non-collapsed at times near $\tau$, providing the claimed contradiction.
\end{proof}
\begin{cor}\label{TanConeNonCollapseCor}
Let $\mathcal{K}=\set{\mu_t}_{t\geq t_0}$ be an integral Brakke flow with bounded area ratios. If there is a $(\yY,\tau)\in\Real^{n+1}\times(t_0,\infty)$ so that a $\mathcal{T}\in\mathrm{Tan}_{(\yY, \tau)}\mathcal{K}$ is (stongly) non-collapsed at time $0$, then $\mathcal{K}$ is (strongly) non-collapsed at time $\tau$. 
\end{cor}
\begin{proof} 
We first show that if $\mathcal{T}$ is non-collapsed at time $0$, then $\mathcal{K}$ is  non-collapsed at time $\tau$.  Indeed, Proposition \ref{NonTermOpenCondProp} implies that there is a sequence $\rho_i\to\infty$ so that for $i$ sufficiently large, $\mathcal{K}^{(\yY,\tau),\rho_i}$ is non-collapsed at time $0$. Hence, the spatial and temporal translation properties of the definition imply that $\mathcal{K}$ is non-collapsed at time $\tau$.

If $\mathcal{T}=\set{\nu_t}_{t\in \Real}$ is strongly non-collapsed at time $0$, then $\mathcal{T}^k= \set{\mu_t\times \mu_{\Real^k}}_{t\in \Real}$ is non-collapsed at time $0$. Moreover, if $\mathcal{K}^k=\set{ \mu_t \times \mu_{\Real^k}}_{t\geq t_0}$, then $\mathcal{T}^k\in\mathrm{Tan}_{((\yY_0,\OO),\tau)} \mathcal{K}^k$. Hence, by what we just showed, $\mathcal{K}^k$ is non-collapsed at time $\tau$.  As $k$ is arbitrary we conclude that $\mathcal{K}$ is strongly non-collapsed at time $\tau$.  
\end{proof}

We conclude this section with the following general structural result:
\begin{prop}\label{GeneralAsympProp}
For $n\geq 2$, if $\mu\in\mathcal{SM}_n(\lambda_{n})$, then one of the following holds:
 \begin{enumerate}
 \item \label{GAP1}$\mu\in\mathcal{CSM}_n(\lambda_{n}$);
 \item \label{GAP2} $\mu$ is strongly non-collapsed; 
 \item \label{GAP3} there is a $\nu\in\mathcal{SM}_n(\lambda_{n})$ with $\lambda[\nu]\leq\lambda[\mu]$ so that $\nu=\hat{\nu}\times\mu_{\Real^{n-k}}$ for some $\hat{\nu}\in\mathcal{CSM}_{k}(\lambda_{n})$ and $1\leq k\leq n-1$.
 \end{enumerate}
\end{prop}
\begin{proof}
 If $n=2$, the result follows from Corollary \ref{NonampleCor}. We now argue by induction. Suppose the proposition holds for all $2\leq k\leq n-1$. If $\mu\in \mathcal{SM}_n(\lambda_n)$ with $\spt(\mu)$ compact, then we are done. If $\spt(\mu)$ is non-compact, the Brakke flow $\mathcal{K}$ associated to $\mu$ satisfies that $X=\set{\yY\in \Real^{n+1}\setminus\{\OO\}:\Theta_{(\yY, 0)}(\mathcal{K})\geq 1}$ is non-empty.  

Pick a point $\yY\in X$ and a tangent flow $\mathcal{T}\in\mathrm{Tan}_{(\yY,0)}\mathcal{K}$. By Lemma \ref{stratificationLem}, up to an ambient rotation, $\mathcal{T}=\set{\hat{\nu}_t\times\mu_\Real}_{t\in\Real}$ with $\hat{\nu}_{-1}\in\mathcal{SM}_{n-1}(\lambda_{n})$ and $\lambda[\hat{\nu}_{-1}]\leq\lambda[\mu]$. In fact, $\hat{\nu}_{-1}\in \mathcal{SM}_{n-1}(\lambda_{n-1})$, as $\lambda_n<\lambda_{n-1}$ by \eqref{StoneCompEqn}. Hence, the induction hypothesis implies that either \eqref{GAP1}, \eqref{GAP2} or \eqref{GAP3} holds for $\hat{\nu}_{-1}$. If $\hat{\nu}_{-1}$ satisfies either $\eqref{GAP1}$ or $\eqref{GAP3}$, then $\mu$ satisfies \eqref{GAP3}. On the other hand, if $\hat{\nu}_{-1}$ satisfies \eqref{GAP2}, then, by definition, $\hat{\nu}_{-1}\times\mu_\Real$ is strongly non-collapsed and so $\mathcal{T}$ is strongly non-collapsed at time $0$. Hence, Corollary \ref{TanConeNonCollapseCor} implies that $\mathcal{K}$ is strongly non-collapsed at time $0$, that is, $\mu$ is strongly non-collapsed.
\end{proof}

Observe that Proposition \ref{SizesingularProp} and \cite[Theorem 0.7]{CIMW} together imply that $\mathcal{CSM}_{2}(\lambda_2)$ is empty. We will give a different proof of this fact in Section \ref{FinalSec}. In fact, our result will be more general, as we will not establish the \emph{a priori} smoothness which would be needed to appeal to \cite[Theorem 0.7]{CIMW} and so also prove that $\mathcal{CSM}_n(\lambda_n)$ is empty for all $2\leq n\leq 6$ .   

\section{Collapsed Singularities of Compact Mean Curvature Flows}\label{NonAmpleSec}
The goal of this section is to show that  every compact boundary measure of finite entropy admits an integral Brakke flow of a special type. Specifically, a Brakke flow which develops a singularity in finite time at which all tangent flows are collapsed at time $0$. There are three steps to the proof. The first is to show that, under a non-fattening condition, the Brakke flow of a canonical boundary motion collapses at the same time it becomes extinct. The second is to use the genericity of the non-fattening condition in order to take limits and so conclude that for any compact boundary measure, there is an integral Brakke flow that collapses at the same time it becomes extinct. The final step is to show that at the extinction time for these Brakke flows, a singularity forms at which all tangent flows are collapsed at time $0$. 

By the extinction time of a Brakke flow, we mean the minimal time at which the support of the flow is empty. More precisely, if $\mathcal{K}=\set{\mu_t}_{t\geq 0}$ is a non-trivial Brakke flow, then the \emph{extinction time} of $\mathcal{K}$ is
\begin{equation}
T_0(\mathcal{K})=\sup\set{t: \spt(\mu_t)\neq\emptyset}.
\end{equation}
When $\spt(\mu_0)$ is compact, the extinction time can be seen to be finite by comparing with the motion of the boundary of a ball containing the support. As general Brakke flows may gratuitously vanish, they need not be collapsed at their extinction time. However, using the maximum principle, it is true that, under a non-fattening condition, the Brakke flow of a canonical boundary motion must be collapsed at its extinction time. 
\begin{lem}\label{BoundaryMotionExtinctTimeLem}
Let $\mu_0$ be a compact boundary measure for which the level-set flow $\mathcal{L}[\spt(\mu_0)]$ is non-fattening. If $(E, \mathcal{K})$ 
is a canonical boundary motion of $\mu_0$ and $T_0=T_0(\mathcal{K})$ is the extinction time of $\mathcal{K}$, then $\mathcal{K}$ is 
collapsed at time $T_0$.
\end{lem}
\begin{proof}
By definition, there is a continuous function $u_0:\Real^{n+1}\to \Real$ and weak solution $u$ to \eqref{LevelsetEqn} with initial condition $u_0$ so that $E=\set{(\xX,t): u(\xX,t)>0}$ and $\mathcal{K}=\set{\mu_t}_{t\geq 0}$ with $\mu_t=\mathcal{H}^n\lfloor\partial^\ast E_t$. Recall that $E_t=\set{\xX: u(\xX,t)>0}$ and $E_t$ is of finite perimeter. As $E_0$ is a non-empty bounded open set, it follows from the avoidance principle, Proposition \ref{AvoidProp}, and the isoperimetric inequality \cite[Theorem 1.28]{Giusti} that the extinction time, $T_0$, satisfies $0<T_0<\infty$ and $E_t$ is empty for $t\geq T_0$. Hence, as the level-set flow of $\spt(\mu_0)$ does not fatten, each $\set{\xX: u(\xX,t)\geq 0}$ for $t\geq T_0$ does not have interior.

Suppose that $\mathcal{K}=\set{\mu_t}_{t\geq 0}$ is non-collapsed at time $T_0$. That is, there is a $(\yY, s)\in \Real^{n+1}\times 
(0,T_0)$ and an $R>0$ so that $\spt(\mu_s)$ separates $B_R(\yY)$ into two components $\Omega_\pm$ containing closed balls 
$\bar{B}_\pm=\bar{B}_{2\sqrt{n(T_0-s)}}(\mathbf{x}_\pm)$. As $B_R(\yY)\cap \spt(\mu_s) \neq 
\emptyset$ and $\mu_s=\abs{D\chi_{E_s}}$, $B_R(\yY)\cap E_s\neq \emptyset$ 
and so, up to relabeling, we may assume that $\Omega_+\cap E_s\neq \emptyset$.
Let $E_s^-=\set{\xX: u(\xX,s)<0}$, we claim $\bar{B}_+\cap E_s^-=\emptyset$. 
To see this, we first note that, as $\mu_s$ is Radon and $\spt(\mu_s)\cap \Omega_+=\emptyset$, $\mu_s(\Omega_+)=0$. 
Applying the Poincar\'{e} inequality for BV 
functions \cite[Lemma 6.4]{Simon} to $\chi_{E_s}\in BV_{loc}(\Omega_+)$, we conclude that, $\chi_{E_s}(\xX)=1$ for a.e. 
$\xX\in\Omega_+$. However,  $E^-_s$ is open and so if $E^-_s\cap\Omega_{+}$ is non-empty, then it has positive Lebesgue measure and so 
$E_s^-\cap\Omega_{+}=\emptyset$, which verifes the claim. Finally, by appealing to the avoidance principle, Proposition \ref{AvoidProp}, we 
conclude that $B_{\sqrt{2n(T_0-s)}}(\mathbf{x}_+)\subset\set{\xX: u(\xX,T_0)\geq 0}$, that is, the latter set has non-empty interior.  This 
contradicts our earlier conclusion.  Thus, $\mathcal{K}$ must be collapsed at its extinction time $T_0$.
\end{proof}

Due to the possibility of fattening, we are not able to ensure the existence of a boundary motion starting from an arbitrary 
compact boundary measure; see \cite[Problem B]{Ilmanen1}. Nevertheless, because the non-fattening condition is generic, we can still 
construct some Brakke flow which is collapsed at the extinction time of the flow.
\begin{prop}\label{NonAmpleProp}
If $\mu_0$ is a compact boundary measure with finite entropy, then there is an integral Brakke flow $\mathcal{K}=\set{\mu_t}_{t\geq 0}$ with bounded area ratios and its extinction time $T_0=T_0(\mathcal{K})>0$ which is collapsed at time $T_0$. Moreover, there is a point $\xX_0\in \Real^{n+1}$ so that $\Theta_{(\xX_0,T_0)}(\mathcal{K})\geq 1$ and all tangent flows to $\mathcal{K}$ at $(\xX_0,T_0)$ are collapsed at time $0$.
\end{prop}
\begin{proof}
If $\mu_0$ is a compact boundary measure with $\lambda[\mu_0]<\infty$, then $\spt (\mu_0)=\partial E_0\neq\emptyset$ and $\mu_0=|D\chi_{E_0}|$ for some non-empty open bounded set $E_0$ of finite perimeter. Fix $\yY$ and $R>r>0$ so that $B_{2\sqrt{2n}r}(\yY)\subset E_0\subset B_{\sqrt{2n}R}(\yY)$. Thus, by Lemma \ref{ApproxLem}, there exists a sequence of bounded open sets $E^i_0$ of finitely many components with $\partial E^i_0$ smooth embedded such that: $B_{\sqrt{2n}r}(\yY)\subset E^i_0\subset B_{2\sqrt{2n}R}(\yY)$, $\chi_{E^i_0}\to\chi_{E_0}$ in $L^1(\Real^{n+1})$ and $\mu^i_0=|D\chi_{E^i_0}|\to |D\chi_{E_0}|=\mu_0$ in the sense of measures. Recall that by  \cite[Theorem 11.3]{Ilmanen1}, the non-fattening condition for level-set flows is generic. Hence, as $\partial E_0^i$ is smooth, we can use the signed distance function to $\partial E_0^i$ to perturb $E_0^i$ in a smooth manner so that the level-set flow $\mathcal{L}(\partial E_0^i)$ satisfies the non-fattening condition. By Theorem \ref{UnitdensityThm}, there exists a canonical boundary motion $(E^i,\mathcal{K}^i)$ of $\mu^i_0$ for each $i$. By comparing with the mean curvature flows of $\partial B_{\sqrt{2n}r}(\yY)$ and $\partial B_{2\sqrt{2n}R}(\yY)$, one verifies that the extinction time $T^i_0$ of the $\mathcal{K}^i$ is contained in the interval $\left[r^2, 4 R^2\right]$.  

As $\mu^i_0\to\mu_0$, for $i$ large enough, $\mu^i_0(\Real^{n+1})\leq 2\mu_0(\Real^{n+1})<\infty$. Hence, the compactness theory of Brakke flows, Theorem \ref{BrakkeFlowCompactThm}, implies that, up to passing to a subsequence and relabeling, the $\mathcal{K}^i=\set{\mu_t^i}_{t\geq 0}$ converges to an integral Brakke flow $\mathcal{K}=\set{\mu_t}_{t\geq 0}$. As $\lambda[\mu_0]<\infty$, it follows from the monotonicity formula in Proposition \ref{MonotoneProp} that $\mathcal{K}$ has bounded area ratios. By passing to a further subsequence, we may assume the extinction times $T^i_0\to T_0\in\left[r^2, 4 R^2\right]$.  

On one hand, the clearing out lemma for Brakke flows \cite[Lemma 6.3]{B}, together with the construction of $\mathcal{K}$ and the fact that $\spt(\mu_t^i)\subset B_{2\sqrt{2n}R}(\yY)$, implies that $\mathcal{K}$ has extinction time $T_0$.  On the other, by Lemma \ref{BoundaryMotionExtinctTimeLem}, the $\mathcal{K}^i$ are collapsed at time $T^i_0$ and so Proposition \ref{NonTermOpenCondProp} implies that $\mathcal{K}$ is also collapsed at time $T_0$. Finally, as $\mathcal{K}$ has extinction time $T_0$,  there exists a sequences $t_i<T_0$ with $t_i\to T_0$ and $\xX_i\in\spt (\mu_{t_i})$ such that $\Theta_{(\xX_i,t_i)}(\mathcal{K})\geq 1$. As $\spt(\mu_t)\subset B_{\sqrt{2n}R}(\yY)$ for all $t$, up to passing to a subsequence and relabeling, $\xX_i\to\xX_0$ and thus, appealing to the upper semi-continuity of Gaussian density, Corollary \ref{UpperctsCor}, we have that $\Theta_{(\xX_0,T_0)}(\mathcal{K})\geq 1$. Hence, as $\mathcal{K}$ is collapsed at time $T_0$, Corollary \ref{TanConeNonCollapseCor} implies that every tangent flow $\mathcal{T}\in\mathrm{Tan}_{(\xX_0,T_0)}\mathcal{K}$ is collapsed at time $0$.
\end{proof}

\section{Entropy Lower Bound}\label{FinalSec}
Observe that \eqref{StoneCompEqn} implies that $\mathcal{CSM}_{k}(\lambda_n)\subset\mathcal{CSM}_{k}(\lambda_k)$ for $2\leq k \leq n-1$. 
\begin{lem}\label{CpctShrinkerInductionLem} 
For $n\geq 2$, if for all $1\leq k\leq n-1$,  $\mathcal{CSM}_{k}(\lambda_{n})$ is empty, then either $\mathcal{CSM}_{n}(\lambda_n)$ is empty or there is a $\mu_0\in \mathcal{CSM}_{n}(\lambda_n)$ so that
\begin{equation}
\lambda[\mu_0]=\inf\set{\lambda[\mu]: \mu\in\mathcal{CSM}_{n}(\lambda_n)}. 
\end{equation}
Furthermore, $\mu_0$ satisfies the following properties:
\begin{enumerate}
 \item $\mu_0$ is a compact boundary measure;
 \item $V_{\mu_0}$ is entropy stable in the sense of \cite[Theorem 0.14]{CM};
 \item $\sing(\mu_0)$ has Hausdorff dimension at most $n-7$.
\end{enumerate}
\end{lem}
\begin{proof}
If $\mathcal{CSM}_{n}(\lambda_n)$ is empty, then we are done. If not, we can define
\begin{equation}
\Lambda_n=\inf\set{\lambda [\mu]: \mu\in\mathcal{CSM}_n(\lambda_n)}.
\end{equation}
Since, for each $\mu\in\mathcal{CSM}_n$, $\spt(\mu)\neq\emptyset$ and $\mu$ has integer multiplicity, $\lambda[\mu]\geq 1$ and so $\Lambda_n\geq 1$. Moreover, by the avoidance principle, Proposition \ref{AvoidProp},  $\spt(\mu)\cap B_{2n}(\OO)\neq\emptyset$.  Hence, if $\mu_i\in\mathcal{CSM}_n(\lambda_n)$ is a minimizing sequence, then Allard's integral compactness theorem (see \cite[Theorem 42.7 and Remark 42.8]{Simon}) implies that, up to passing to a subsequence and relabeling, $\mu_i\to\mu_0$ for $\mu_0\in\mathcal{SM}_n$ with $1\leq \lambda[\mu_0]\leq\Lambda_n$. 

We first show that the assumption that $\mathcal{CSM}_{k}(\lambda_{n})=\emptyset$ for all $1\leq k\leq n-1$ implies that $\mu_0\in\mathcal{CSM}_{n}(\lambda_n)$. Indeed, by Proposition \ref{OrientabilityLowEnt}, the $\mu_i$ are compact boundary measures and so are collapsed by Lemma \ref{NoCpctAmpleShrinkerLem}. Thus, $\mu_0$ is collapsed by Proposition \ref{NonTermOpenCondProp} applied to the associated Brakke flows. Hence, Proposition \ref{GeneralAsympProp} together with our hypothesis implies that $\mu_0\in\mathcal{CSM}_n(\lambda_n)$ and so $\lambda[\mu_0]=\Lambda_n$.

We next show that $V_{\mu_0}$ is entropy stable. Indeed, by Propositions \ref{SizesingularProp} and \ref{OrientabilityLowEnt}, $\mu_0$ is a compact boundary measure with $\mathcal{H}^{n-2}(\sing(\mu_0))=0$. Let $X$ be any compactly supported vector field on $\Real^{n+1}$ with $\spt(X)\cap\sing(\mu_0)=\emptyset$. Denote the flow of $X$ by $\tau\mapsto\phi^\tau$ and set $\mu^\tau_0=\phi^\tau_*\mu_0$. Clearly, the $\mu^\tau_0$ are compact boundary measures and we may assume that the $\mu^\tau_0$ are of finite entropy (otherwise, we are done). Thus, Proposition \ref{NonAmpleProp} gives integral Brakke flows $\mathcal{K}^\tau=\set{\mu_{t}^\tau}_{t\geq 0}$ with bounded area ratios, and points $(\yY^\tau,s^\tau)\in\Real^{n+1}\times\Real^+$ so that the tangent flows to $\mathcal{K}^\tau$ at $(\yY^\tau,s^\tau)$ are collapsed at time $0$. Invoking Proposition \ref{GeneralAsympProp} again, the entropy of the time $-1$ slice of these tangent flows is bounded from below by $\Lambda_n$. Thus, the monotonicity formula in Proposition \ref{MonotoneProp} implies that $\lambda[\mu^\tau_0]\geq\Lambda_n$. Hence, $V_{\mu_0}$ is entropy stable and we may apply \cite[Theorem 0.14]{CM} to conclude that $\sing(\mu_0)$ has Hausdorff dimension at most $n-7$.
\end{proof}

Using Lemma \ref{CpctShrinkerInductionLem}, it is easy to establish the following non-existence result.
\begin{prop}\label{ShrinkerMinimizersVarifoldProp} 
If $2\leq n\leq 6$, then $\mathcal{CSM}_{n}(\lambda_n)$ is empty. If $n=7$ and $\mathcal{CSM}_{7}(\lambda_7)$ is non-empty, then there is a $\mu_0\in \mathcal{CSM}_{7}(\lambda_7)$ so that
\begin{equation}
\lambda[\mu_0]=\inf\set{ \lambda[\mu]: \mu\in\mathcal{CSM}_{7}(\lambda_7)}. 
\end{equation}
Furthermore, $\mu_0$ satisfies the following properties:
\begin{enumerate}
 \item $\mu_0$ is a compact boundary measure;
 \item $V_{\mu_0}$ is entropy stable in the sense of \cite[Theorem 0.14]{CM};
 \item $\sing(\mu_0)$ consists of a non-empty finite set of points.
\end{enumerate}
\end{prop}
\begin{rem}
The existence of singular points when $n=7$ is due to the possible existence of entropy stable cones in $\Real^8$ with entropy less than $\lambda_7$.
\end{rem}
\begin{proof}
We argue by induction. If $n=2$, then the hypothesis of Lemma \ref{CpctShrinkerInductionLem} holds by direct computation. If $\mathcal{CSM}_2(\lambda_2)$ is non-empty, there exists a $\mu_0=\mu_\Sigma\in\mathcal{CSM}_2(\lambda_2)$ with $\Sigma$ an entropy stable closed self-shrinker. Thus, \cite[Theorem 0.12]{CM} implies that $\Sigma=\rho\mathbb{S}^2+\yY$ and so $\lambda[\mu_0]=\lambda_2$, which is a contradiction. Hence, $\mathcal{CSM}_2(\lambda_2)$ is empty. Arguing inductively, Lemma \ref{CpctShrinkerInductionLem} and \cite[Theorem 0.14]{CM} imply that $\mathcal{CSM}_n(\lambda_n)$ is empty for $2\leq n\leq 6$.  

For $n=7$, if $\mathcal{CSM}_7(\lambda_7)$ is non-empty, there exists an entropy stable $\mu_0\in\mathcal{CSM}_7(\lambda_7)$. Hence, by the classification of entropy stable self-shrinkers in \cite[Theorem 0.12]{CM}, $\sing (\mu_0)$ must be non-empty. Given $\yY\in\sing (\mu_0)$, any tangent cone of $V_{\mu_0}$ at $\yY$ is a stable stationary cone in $\Real^8$ with singular set of codimension at least two. Therefore, it follows from the regularity theorem in \cite{SchoenSimon} that the cone has an isolated singularity and so $\sing (\mu_0)$ is discrete. 
\end{proof}

We have the following consequence of Proposition \ref{ShrinkerMinimizersVarifoldProp}.
\begin{cor} \label{MainCor}
If $2\leq n\leq 6$ and $\mu\in\IM_n(\Real^{n+1})$ is a compact boundary measure, then $\lambda [\mu]\geq\lambda_n$ with equality if and only if $\mu=\mu_{\rho\mathbb{S}^n+\yY}$ for some $\rho>0$ and $\yY\in\Real^{n+1}$.
\end{cor}
\begin{proof}
Suppose that $\mu\in\IM_n(\Real^{n+1})$ is a compact boundary measure and $\lambda [\mu]\leq\lambda_n$. By Proposition \ref{NonAmpleProp}, there exists an integral Brakke flow $\mathcal{K}$ with bounded area ratios and starting from $\mu$, and a point $(\yY,s)\in \Real^{n+1}\times\Real^+$ so that any tangent flow $\mathcal{T}=\set{\nu_t}_{t\in\Real}$ in $\mathrm{Tan}_{(\yY,s)}\mathcal{K}$ is collapsed at time $0$. By the lower semi-continuity of entropy, we have that $\lambda[\nu_{-1}]\leq\lambda[\mu]\leq\lambda_n$.

 We claim that $\nu_{-1}$ has compact support. For $n=2$, this claim follows directly from Corollary \ref{NonampleCor}. For $3\leq n\leq 6$ we argue by contradiction.  If $\nu_{-1}$ does not have compact support, then, by Lemma \ref{stratificationLem} and Corollary \ref{TanConeNonCollapseCor}, there is a $\hat{\nu}\in\mathcal{SM}_{n-1}$ so that $\lambda[\hat{\nu}]\leq\lambda[\nu_{-1}]\leq\lambda_n<\lambda_{n-1}$ and $\hat{\nu}\times \mu_\Real$ is collapsed. As $\hat{\nu}$ is not strongly non-collapsed, it follows from Proposition \ref{GeneralAsympProp} and the fact that $n\geq 3$ that either $\hat{\nu}\in \mathcal{CSM}_{n-1}(\lambda_{n-1})$ or there is a $\tilde{\nu}\in \mathcal{CSM}_k(\lambda_{n-1})$ for $k<n-1$. In either case, this contradicts Proposition \ref{ShrinkerMinimizersVarifoldProp} and verifies the claim. Thus, invoking again Proposition \ref{ShrinkerMinimizersVarifoldProp}, $\lambda[\nu_{-1}]=\lambda_n$ and, furthermore, by the monotonicity formula, the entropy is invariant along the flow. In particular, $\mathcal{K}$ is self-similar with respect to $(\yY,s)$ and $\mu^{\yY,1/\sqrt{s}}\in\mathcal{CSM}_n$ with $\lambda[\mu]=\lambda_n$. Hence it remains only to characterize the case of equality.

Notice that any deformation of $\reg(\mu)$ by a vector field gives a new compact boundary measure. Hence, as we have just shown, it is impossible to construct such deformations to decrease the entropy. As Proposition \ref{SizesingularProp} implies that $\mathcal{H}^{n-2}(\sing(\mu))=0$, we conclude that $V_{\mu}$ is entropy stable and so $\mu=\mu_{\sqrt{s}\, \mathbb{S}^n+\yY}$ by \cite[Theorem 0.14]{CM}.
\end{proof}

As closed hypersurfaces separate $\Real^{n+1}$, Theorem \ref{MainThm} follows by applying Corollary \ref{MainCor} to $\mu=\mu_\Sigma$ for any $\Sigma$ closed hypersurface in $\Real^{n+1}$ for $2\leq n\leq 6$. 

Our methods give an easy proof, for $2\leq n \leq 6$, of the existence of an entropy gap within $\mathcal{CSM}_n$ around $\mu_{\mathbb{S}^n}$. This is true in all dimensions when one considers only smooth self-shrinkers; see \cite[Theorem 0.6]{CIMW}.
\begin{cor}\label{EntGapCor}
There is an $\epsilon=\epsilon(n)>0$ so that if $2\leq n\leq 6$ and $\mu\in \mathcal{CSM}_n(\lambda_n+\epsilon_n)$, then $\mu=\mu_{\mathbb{S}^n}$.
\end{cor}
\begin{proof}
If there is no such $\epsilon_n$, then there is a sequence $\mu_i\in \mathcal{CSM}_n$ with $\lambda[\mu_i]>\lambda_n$ and so that $\lambda[\mu_i]\to\lambda_n<\frac{3}{2}$. By the same arguments as in the first paragraph of the proof of Lemma \ref{CpctShrinkerInductionLem}, up to  passing to subsequence and relabeling, the $\mu_i\to\mu\in\mathcal{SM}_n$ with $\lambda[\mu]\leq\lambda_n<\frac{3}{2}$. By Proposition \ref{OrientabilityLowEnt} and Lemma \ref{NoCpctAmpleShrinkerLem}, the $\mu_i$ are collapsed. Thus, Proposition \ref{NonTermOpenCondProp} implies that $\mu$ is also collapsed. Hence, arguing as in the proof of Corollary \ref{MainCor}, it follows that $\mu$ has compact support. Invoking Proposition \ref{OrientabilityLowEnt} again, $\mu$ is a compact boundary measure and so, by Corolloary \ref{MainCor}, $\mu=\mu_{\mathbb{S}^n}$. It follows from Allard's regularity theorem (see \cite[Theorem 24.2]{Simon}) that $\spt(\mu_i)\to\mathbb{S}^n$ in $C^\infty_{loc}(\Real^{n+1})$. Hence, for $i$ sufficiently large, $\spt(\mu_i)\cap\partial B_{2n}(\OO)=\emptyset$ and so by the avoidance principle, Proposition \ref{AvoidProp}, for such $i$, $\spt(\mu_i)\subset B_{2n}(\OO)$. The smooth convergence implies that, for $i$ sufficiently large, $\mu_i=\mu_{\Sigma_i}$ for a closed strictly convex self-shrinker $\Sigma_i$. By \cite[Theorem 4.1]{Huisken}, $\Sigma_i=\mathbb{S}^n$ and so $\lambda[\mu_i]=\lambda_n$. This contradiction proves the claim.
\end{proof}

Finally, we observe that Corollary \ref{MainCor} implies an entropy lower bound for certain non-compact self-shrinkers.
\begin{defn}\label{partiallycollapsedDefn}
 A $\mu\in \mathcal{SM}_n$ with $\spt(\mu)$ non-compact is \emph{partially collapsed} if there is $\yY\neq\OO$ so that if $\mathcal{K}$ is the associated Brakke flow to $\mu$, then $\Theta_{(\yY,0)}(\mathcal{K})\geq 1$ and some tangent flow $\mathcal{T}\in\mathrm{Tan}_{(\yY,0)}\mathcal{K}$ is collapsed at time $0$.
\end{defn}
This is a weaker notion than being collapsed. For instance,  the measure of a self-shrinker with one end asymptotic to a cylinder and another asymptotic to a smooth cone would be partially collapsed but not collapsed.

\begin{cor}\label{EntLowBndCollapsedCor}
For $3\leq n\leq 7$, if $\mu\in\mathcal{SM}_n$ has non-compact support and is partially collapsed, then $\lambda[\mu]\geq\lambda_{n-1}$ with equality if and only if, up to an ambient rotation, $\mu=\mu_{\mathbb{S}^{n-1}\times\Real}$.
\end{cor}
\begin{proof}
Let $\mu\in\mathcal{SM}_n$ with non-compact support and $\mu$ is partially collapsed. Assume that $\lambda[\mu]\leq\lambda_{n-1}$. Consider the associated Brakke flow $\mathcal{K}$ to $\mu$. Then there is a point $\yY\neq\OO$ so that $\Theta_{(\yY,0)}(\mathcal{K})\geq 1$ and a $\mathcal{T}\in\mathrm{Tan}_{(\yY,0)}\mathcal{K}$ is collapsed at time $0$. By Lemma \ref{stratificationLem}, up to an ambient rotation, $\mathcal{T}=\set{\nu_t}_{t\in\Real}$ splits off a line. That is, $\nu_t=\hat{\nu}_t\times\mu_{\Real}$, where $\set{\hat{\nu}_t}_{t\in\Real}$ is the Brakke flow associated to $\hat{\nu}_{-1}\in\mathcal{SM}_{n-1}$. As $\mathcal{T}$ is collapsed at time $0$, both $\nu_{-1}$ and $\hat{\nu}_{-1}$ are collapsed. By the lower semi-continuity of entropy, $\lambda[\hat{\nu}_{-1}]\leq\lambda[\mu]\leq\lambda_{n-1}$. Note that the same argument as in the proof of Corollary \ref{MainCor} gives that $\hat{\nu}_{-1}$ has compact support. Thus, by Proposition \ref{ShrinkerMinimizersVarifoldProp} and Corollary \ref{MainCor}, $\lambda[\hat{\nu}_{-1}]=\lambda_{n-1}$ and $\hat{\nu}_{-1}=\mu_{\mathbb{S}^{n-1}}$. Hence it follows from the monotonicity formula in Proposition \ref{MonotoneProp} that $\mathcal{T}=\mathcal{K}$, i.e., $\mu=\mu_{\mathbb{S}^{n-1}\times \Real}$ as claimed.
\end{proof}

For $3\leq n\leq 7$, Theorem \ref{CollapsedShrinkers} directly follows from Corollary \ref{EntLowBndCollapsedCor} applied to $\mu=\mu_\Sigma$. However, when $n=2$, some care has to be taken as $\lambda_1>3/2$. Nevertheless, because $\Sigma$ is smooth, the result will follow by using the work of White \cite{WhiteMod2}.
\begin{proof}[Proof of Theorem \ref{CollapsedShrinkers} when $n=2$]
Suppose that $\lambda[\Sigma]\leq\lambda_1$. As $\Sigma$ is complete smooth embedded, the Brakke flow $\mathcal{K}$ associated to $\mu_\Sigma$ is cyclic mod $2$ in the sense of \cite[Definition 4.1]{WhiteMod2}. Hence, by \cite[Theorem 4.2]{WhiteMod2}, every tangent flow to $\mathcal{K}$ is cyclic mod $2$. In particular, exactly as in the proof of Corollary \ref{EntLowBndCollapsedCor}, one can use Lemma \ref{stratificationLem} to construct $\hat{\nu}_{-1}\in\mathcal{SM}_1$ which is collapsed, satisfies $\lambda[\hat{\nu}_{-1}]\leq\lambda_1<2$, and is such that $\partial [V_{\hat{\nu}_{-1}}]=0$. Here $[V_{\hat{\nu}_{-1}}]$ is the rectifiable mod $2$ flat chain associated to the integral varifold $V_{\hat{\nu}_{-1}}$; see \cite{WhiteMod2} for the specifics. A consequence of this last fact is that all tangent cones to $V_{\hat{\nu}_{-1}}$ consists of unions of \emph{even} numbers of rays. In particular, as $\lambda[\hat{\nu}_{-1}]<2$, $\hat{\nu}_{-1}=\mu_{\gamma}$ for some complete self-shrinker $\gamma\subset\Real^2$. By the classification of complete self-shrinkers in \cite{AL}, $\gamma$ must be either $\mathbb{S}^1$ or $\mathbb{R}^1$.  As the latter is non-collapsed, $\hat{\nu}_{-1}=\mu_{\mathbb{S}^1}$. Then, following the same argument as in Corollary \ref{EntLowBndCollapsedCor}, $\mathcal{K}$ is self-similar and, up to an ambient rotation, $\Sigma=\Real\times\mathbb{S}^1$.
\end{proof}

\appendix
\section{Proof of Lemma \ref{ApproxLem}}\label{ProofApprox}
In this appendix, we present a complete proof of Lemma \ref{ApproxLem} -- while this is standard, we could not find a reference in the literature. First, it follows from \cite[Equation (1.12) and Theorem 1.24]{Giusti} that there exists a sequence of bounded open sets $E_j$ of finitely many components with $\partial E_j$ smooth embedded such that: $B_{r}(\xX)\subset E_j\subset B_{2R}(\xX)$, $\chi_{E_j}\to\chi_E$ in $L^1(\Real^{n+1})$ and 
\begin{equation}
\lim_{j\to\infty}\int\abs{D\chi_{E_j}}=\int\abs{D\chi_E}.
\end{equation}

Next, fix any $f\in C^0_c(\Real^{n+1},\Real^{\geq 0})$. By the lower semi-continuity, 
\begin{equation}\label{LowcontEqn}
\liminf_{j\to\infty}\int f\abs{D\chi_{E_j}}\geq\int f\abs{D\chi_E}.
\end{equation}
On the other hand, setting $M=\norm{f}_{C^0}$,
\begin{equation}
\begin{split}
\int M\abs{D\chi_E} & =\lim_{j\to\infty}\int M\abs{D\chi_{E_j}}=\lim_{j\to\infty} \int f\abs{D\chi_{E_j}}+\int (M-f)\abs{D\chi_{E_j}}\\
& \geq\limsup_{j\to\infty}\int f\abs{D\chi_{E_j}}+\liminf_{j\to\infty}\int (M-f)\abs{D\chi_{E_j}}\\
& \geq\limsup_{j\to\infty} \int f\abs{D\chi_{E_j}}+\int (M-f)\abs{D\chi_E}.
\end{split}
\end{equation}
In the last inequality above, we observe that
\begin{equation}
\int (M-f)\abs{D\chi_{E_j}}=\int\phi (M-f)\abs{D\chi_{E_j}},
\end{equation}
where $\phi$ is chosen to be a cut-off function with $\phi=1$ on a sufficiently large ball containing $E_j$ and $E$, and then appeal to \eqref{LowcontEqn}. Hence,
\begin{equation}\label{UppercontEqn}
\limsup_{j\to\infty}\int f\abs{D\chi_{E_j}}\leq\int f\abs{D\chi_E}.
\end{equation}

Therefore, combining \eqref{LowcontEqn} and \eqref{UppercontEqn} gives that for all $f\in C_c^0(\Real^{n+1},\Real^{\geq 0})$,
\begin{equation}
\lim_{j\to\infty}\int f\abs{D\chi_{E_j}}=\int f\abs{D\chi_E},
\end{equation}
that is, $\abs{D\chi_{E_j}}\to\abs{D\chi_E}$ in the sense of Radon measures.

\begin{acknowledgement*} 
The second author would like to thank Richard Bamler and Brian White for helpful conversations regarding the proof of Proposition \ref{OrientabilityLowEnt}. She is also very grateful to Neshan Wickramasekera for inviting her to visit CMS of Cambridge University in Summer 2013 where this project was initiated.
\end{acknowledgement*}

\end{document}